\documentclass[a4paper,12pt]{amsart}

\usepackage{amsfonts}
\usepackage{amsmath}
\usepackage{amssymb}

\usepackage{mathrsfs}
\usepackage{hyperref}

\numberwithin{equation}{section}
\theoremstyle{plain}
\newtheorem{thm}{Theorem}[section]
\newtheorem{prop}[thm]{Proposition}
\newtheorem{cor}[thm]{Corollary}

\theoremstyle{definition}

\newtheorem{rem}[thm]{Remark}

\newcommand{\be}{\begin{equation}}
\newcommand{\ee}{\end{equation}}

\def\R{{\mathbb R}}
\def\D{{\mathbb D}}

\def\N{{\mathbb N}}

\def\C{{\mathbb C}}
\def\S3{{{\mathbb S}^3}}
\def\SU2{{{\rm SU}(2)}}
\def\Rn{{{\mathbb R}^n}}

\def\Gh{{\widehat{G}}}

\def\HS{{\mathtt{HS}}}

\def\p#1{{\left({#1}\right)}}

\def\jp#1{{\left\langle{#1}\right\rangle}}

\def\Dcal{{\mathcal D}}
\def\Lcal{{\mathcal L}}

\DeclareMathOperator{\Tr}{Tr}

\def\Rr{{\mathbb R}}

\def\C{{\mathbb C}}

\def\Rn{{\mathbb R}^n}
\def\R2n{{\mathbb R}^{2n}}

\def\S{{\mathcal S}}
\def\D{{\mathcal D}}

\def\Rr{{\mathbb R}}
\def\Rn{{\mathbb R}^n}

\def\C{{\mathbb C}}

\def\R2{{\mathbb R}^2}
\def\R2n{{\mathbb R}^{2n}}

\def\S{{\mathcal S}}
\def\D{{\mathcal D}}

\def\sumxi{\sum_{[\xi]\in\Gh}}

\def\T1{{\mathbb T}^1}
\def\St{{\mathbb S}^3}

\def\dxi{{d_\xi}}
\def\Cdxi{{\C^{\dxi\times\dxi}}}

\newcommand{\eps}{\varepsilon}

\newcommand{\esp}{\mathrm{e}}

\begin{document}

\title[Wave equation for sums of squares on compact Lie groups]
{Wave equation for sums of squares on compact Lie groups}

\author[Claudia Garetto]{Claudia Garetto}
\address{
  Claudia Garetto:
  \endgraf
  Department of Mathematical Sciences
  \endgraf
  Loughborough University
  \endgraf
  Loughborough, Leicestershire, LE11 3TU
  \endgraf
  United Kingdom
  \endgraf
  {\it E-mail address} {\rm c.garetto@lboro.ac.uk}
  }
\author[Michael Ruzhansky]{Michael Ruzhansky}
\address{
  Michael Ruzhansky:
  \endgraf
  Department of Mathematics
  \endgraf
  Imperial College London
  \endgraf
  180 Queen's Gate, London SW7 2AZ
  \endgraf
  United Kingdom
  \endgraf
  {\it E-mail address} {\rm m.ruzhansky@imperial.ac.uk}
  }

\thanks{The second
 author was supported by the EPSRC
 Leadership Fellowship EP/G007233/1 and by EPSRC Grant EP/K039407/1.}
\date{\today}


\begin{abstract}
In this paper we investigate the well-posedness of the Cauchy problem for the
wave equation for sums of squares of vector fields on compact Lie groups. 
We obtain the loss of regularity
for solutions to the Cauchy problem in local Sobolev spaces 
depending on the order to which the H\"ormander condition is 
satisfied, but no loss in globally defined spaces. 
We also establish Gevrey well-posedness for equations with
irregular coefficients and/or multiple characteristics. As in the Sobolev spaces,
if formulated in local coordinates,
we observe well-posedness with the loss of local Gevrey order depending on the
order to which the H\"ormander condition is satisfied.
\end{abstract}

\maketitle

\section{Introduction}

In this paper we investigate the well-posedness of the Cauchy problem for 
time-dependent wave equations associated to sums of squares of invariant vector fields 
on compact Lie groups. Such analysis is motivated, in particular, by general investigations 
of the well-posedness and wave propagation governed by subelliptic operators and problems
with multiplicities.
An often encountered example of subelliptic behaviour is a sum of squares of vector fields,
extensively analysed by H\"ormander \cite{hormander_67,Hormander:hypoellipticity-1967},
Oleinik and Radkevich \cite{Oleinik-Radkevich:BK-1973}, Rothschild and Stein
\cite{Rothschild-Stein:AM-1976}, and by many others. For invariant operators on
compact Lie groups, the sum of squares becomes formally self-adjoint,
making the corresponding wave equation hyperbolic, a necessary condition for
the analysis of the corresponding Cauchy problem. Already in this setting,
we discover a new phenomenon of the loss of the local Gevrey regularity for its
solutions. Moreover, this loss is linked to the order to which the
H\"ormander condition is satisfied.

Thus, let $G$ be a compact Lie group
of dimension $n$ with Lie algebra $\mathfrak g$, and let
$X_1,\ldots,X_k$ be a family of left-invariant vector fields in $\mathfrak g$.  
Let 
\begin{equation}\label{EQ:subL}
\Lcal:=X_1^2+\cdots+X_k^2
\end{equation}
be the sum of squares of derivatives defined by the vector fields.
If the iterated commutators of $X_1,\ldots,X_k$ 
span the Lie algebra of $G$,
the operator $\Lcal$ is a sub-Laplacian on $G$,
hypoelliptic in view of H\"ormander's sum of the squares theorem.

With or without the H\"ormander condition,
it can be shown that the operator $\partial_t^2-\Lcal$ is (weakly) hyperbolic
(see Remark \ref{REM:hyp}).
For a continuous function $a=a(t)\geq 0$, we will be concerned with the
Cauchy problem
\begin{equation}\label{CPa}
\left\{ \begin{split}
\partial_{t}^{2}u(t,x)-a(t)\Lcal u(t,x)&=0, \; (t,x)\in [0,T]\times G,\\
u(0,x)&=u_{0}(x), \; x\in G, \\
\partial_{t}u(0,x)&=u_{1}(x), \; x\in G.
\end{split}
\right.
\end{equation}
When localised, the Cauchy problem \eqref{CPa} is a weakly hyperbolic
equation with both time and space dependent coefficients, and the available
results and techniques are rather limited compared to, for example, the situation when
the coefficients depend only on time. 
For example, general Gevrey well-posedness
results of Bronshtein \cite{Bronshtein:TMMO-1980} or Nishitani
\cite{Nishitani:BSM-1983} may apply for some $a$ and $\Lcal$, but in general they
do not take into account the 
geometry of the problem and of the operator $\Lcal$. 

In the case of the Euclidean space $\Rn$, the Cauchy problem for the operator
$\partial_t^2-a(t)\Delta$ with the Laplacian $\Delta$ has been extensively studied.
It is known that the Cauchy problem for this operator may be not well-posed in
$C^\infty(\Rn)$ and in $\Dcal'(\Rn)$ if the function $a(t)$ becomes zero or is irregular,
see, respectively, Colombini and Spagnolo
\cite{Colombini-Spagnolo:Acta-ex-weakly-hyp}, and
Colombini, Jannelli and Spagnolo \cite{Colombini-Jannelli-Spagnolo:Annals-low-reg}.
Thus, Gevrey spaces appear naturally in such well-posedness problems already on $\Rn$,
and for the latter equation, a number of sharp well-posedness results in Gevrey
spaces have been established by Colombini, de Giorgi and Spagnolo
\cite{Colombini-deGiordi-Spagnolo-Pisa-1979}.

Our analysis will cover the case of the Laplacian $\Delta$ on the compact Lie 
group $G$ since we can write it
as $\Lcal=X_1^2+\cdots+ X_n^2$ for a basis $X_1,\ldots,X_n$ of the Lie algebra
of $G$. For different ways of representing Laplacians on compact Lie groups we
refer to an extensive discussion in Stein \cite{Stein:BOOK-topics-Littlewood-Paley}.
In the case of the Laplacian we recover the orders that can be obtained from
the work of Nishitani \cite{Nishitani:BSM-1983} since in this case we can write
$\Lcal$ in local coordinates in the divergence form. For sub-Laplacian $\Lcal$
this no longer applies (neither are the results of Jannelli \cite{Jannelli:JMKU-1984}
because of the lack of divergence form and appearing lower order terms).

Also in the case of the Laplacian, the strictly hyperbolic wave and Schr\"odinger equations 
on compact Lie groups have been recently analysed in the framework of the KAM theory by 
Berti and Procesi \cite{Berti-Procesi:Nonlinear-wave-Sch-DMJ-2011},
with further additions of nonlinear terms. We can refer to their paper and references therein,
as well as to Helgason \cite{Helgason:wave-eqns-hom-spaces-1984},
for a thorough explanation of the appearance of these partial differential 
equations on compact Lie groups, and the need to study them contributed greatly 
to the development of the modern theory of
compact Lie groups, starting with Weyl \cite{Weyl:origins}.
$L^p$-estimates for wave equations have been considered on Lie groups as well.
Here, in the case of the Laplacian on compact Lie groups,
the loss of regularity in $L^p$ has been obtained by Chen, Fan and Sun 
\cite{Chen-Fan-Sung:WE-cpt-Lie-gps-JFA-2010}; however, for $p>1$, this loss
can be also deduced from the localised $L^p$-estimates for Fourier integral
operators by Seeger, Sogge and Stein \cite{SSS:1991}.
In turn, different techniques are required for the wave equation with sub-Laplacians,
see e.g. the case of the standard sub-Laplacian on the Heisenberg group
by M\"uller and Stein \cite{Muller-Stein:Lp-wave-Heis}.
The finite propagation speed results for wave equations for subelliptic operators
are also known in different related settings: analysis for abstract operators was developed by
Melrose \cite{Melrose:wave-subelliptic-1986}, with explicit
formulae for the wave kernels on the Heisenberg 
group obtained by Taylor 
\cite{Taylor:BK-Noncomm-harmonic-analysis} and Nachman
\cite{Nachman:wave-Heisenberg-CPDE-1982} and,
more recently, by Greiner, Holcman and Kannai
\cite{Greiner-et-al:WE-second-order-DMJ-2002}.

The point of this paper is that by applying
the global Fourier analysis on $G$ to the Cauchy problem \eqref{CPa} we can 
view it as an equation with coefficients depending only on $t$, leading to a range 
of sharp results depending on further properties of the function $a(t)$. However,
since the global Fourier coefficients of functions on $G$ become matrix valued, 
on the Fourier transform side the scalar equation \eqref{CPa} becomes a system,
with the size of the system going to infinity with the dimension of representations, 
unless $G$ is a torus.
An important observation enabling our analysis is that we can explore the sum of squares
structure of the operator $\Lcal$ using a notion of a matrix symbol for operators
on compact Lie groups. Thus, we will show that the system for Fourier coefficients
decouples completely into independent scalar equations for the matrix components
of Fourier coefficients. The equations are determined by the entries of the matrix
symbol of $\Lcal$ which we also study for this purpose, in particular establishing
lower bounds for its eigenvalues in terms of the order to which the H\"ormander
condition is satisfied (in the case when it is indeed satisfied).

Our results will apply to general operators $\Lcal$ of the form \eqref{EQ:subL}
without $X_1,\ldots,X_k$ necessarily
satisfying the H\"ormander condition. However, if the H\"ormander 
condition is satisfied, the well-posedness of \eqref{CPa} in $C^\infty(G)$,  
$\Dcal'(G)$, or usual Gevrey spaces on $G$ viewed as a manifold will follow.
Moreover, such well-posedness statements can be refined with respect to the loss of
regularity and the orders
of appearing Sobolev or Gevrey spaces if we know also the order to which the
H\"ormander condition is satisfied. To our knowledge, this phenomenon of local loss
of Gevrey regularity appears
to be new in the study of weakly hyperbolic equations.

Let us give an example of such an equation \eqref{CPa} on the 3-sphere $G=\St$. 
Here, if $X,Y,Z$ are an orthonormal basis (with respect to the Killing form)
of left-invariant vector fields on $\St$, then we can set $\Lcal$ to be the sub-Laplacian
\begin{equation}\label{EQ:subLap}
\Lcal:=\Lcal_{\St, sub}:=X^{2}+Y^{2}.
\end{equation}
Here, we can view the 3-sphere $\St$ as a Lie group with respect to the quaternionic 
product of $\Rr^{4}$, and note that it is globally diffeomorphic and isomorphic to the
group $\SU2$ of unitary $2\times 2$ matrices of determinant one,
with the usual matrix product.
We also note that in 
Euler's angles $(\phi,\psi,\theta)$ the sub-Laplacian $\Lcal_{\St, sub}$ 
has the form
$$
\Lcal_{\St, sub} =\frac{1}{\sin^2\theta}\partial_\phi^2-
2\frac{\cos\theta}{\sin^2\theta}\partial_\phi\partial_\psi+
\left(\frac{1}{\sin^2\theta}-1\right)\partial_\psi^2+
\partial_\theta^2+\frac{\cos\theta}{\sin\theta}\partial_\theta,
$$
see e.g. \cite[Section 11.9]{Ruzhansky-Turunen:BOOK}, where we can take the almost injective
range for Euler angles $0\leq\phi<2\pi$, $0<\theta<\pi$, $-2\pi\leq\psi<2\pi$
(see \cite[Section 11.3]{Ruzhansky-Turunen:BOOK}).
Denoting by $\eta$ the dual variables to Euler's angles,
we can see that the principal symbol of $\Lcal_{\St, sub}$
in these coordinates is $\frac{1}{\sin^2\theta} (\eta_1-\cos \theta\, \eta_2)^2+\eta_3^2$
so that the equation \eqref{CPa} is weakly hyperbolic, with multiplicities on the set
$\eta_1=\cos\theta\, \eta_2$, $\eta_3=0$, even if $a(t)\equiv 1$.

Other examples of sub-Laplacians of different steps can be constructed from
the lists of roots systems (see e.g. Fegan \cite[Chapter 8]{Fegan:bk-compact-Lie-groups}),
although there are certain limitations on possible root strings,
see e.g. Knapp \cite[Section II.5]{Knapp:bk-Lie-groups-beyond-intro}.

The paper is organised as follows. In Section \ref{SEC:results}
we will formulate our results. In Section \ref{SEC:Prelim}
we will set the notation for our approach and will establish properties of
matrix symbols and Sobolev spaces associated to sub-Laplacians.
In Section \ref{SEC:reduction} we will give proofs of our results.

The authors would like to thank V\'eronique Fischer for stimulating discussions
and Ferruccio Colombini for comments.

\section{Main results}
\label{SEC:results}

Thus, for this paper, as in the introduction, we let $G$ be a compact Lie group
and let $X_1,\ldots,X_k$ be a family of left-invariant vector fields in $\mathfrak g$.  
Then we fix the operator $\Lcal$ as in \eqref{EQ:subL}.

In our results below, concerning the Cauchy problem \eqref{CPa}, we will aim at
carrying out comprehensive analysis and distinguish between the following cases:
\begin{itemize}
\item[] {\bf Case 1:} $a(t)\ge a_0>0$, $a\in C^1([0,T])$;
\item[] {\bf Case 2:} $a(t)\ge a_0>0$, $a\in C^\alpha([0,T])$, with $0<\alpha<1$;
\item[] {\bf Case 3:} $a(t)\ge 0$, $a\in C^\ell([0,T])$ with $\ell\ge 2$;
\item[] {\bf Case 4:} $a(t)\ge 0$, $a\in C^\alpha([0,T])$, with $0<\alpha<2$.
\end{itemize}
Thus, Case 1 is the regular time-non-degenerate case when we obtain the well-posedness
in Sobolev spaces associated to the operator $\Lcal$. If $\Lcal$ is hypoelliptic with
H\"ormander condition satisfied to order $r$ we show the loss in local regularity
depending on $r$. Case 2 is devoted to non-zero $a(t)$ but allowing it to be of
H\"older regularity $\alpha$ only. Cases 3 and 4 are devoted to the situation when
there may be also degeneracies with respect to $t$, in both situations when $a(t)$
is regular and not. The threshold $\alpha=2$ is natural from the point of view that 
in general, if $a\in C^\alpha([0,T])$ with $0<\alpha<2$, the characteristic roots
are in H\"older spaces $C^{\frac{\alpha}{2}}([0,T])$ with $0<\frac{\alpha}{2}<1$,
thus providing a similar setting to that in Case 2. Thus, the proofs in Cases 2 and 4
will be similar and based on the regularisation and separation of characteristic
roots. Case 3 will rely on construction of a quasi-symmetriser, while in Case 1 
a symmetriser will suffice.

For any $s\in\Rr$, we define Sobolev spaces $H^s_\Lcal(G)$ associated to 
$\Lcal$ by 
\begin{equation}\label{EQ:HsL0}
H^s_\Lcal(G):=\left\{ f\in\Dcal'(G): (I-\Lcal)^{s/2}f\in L^2(G)\right\},
\end{equation}
with the norm $\|f\|_{H^s_\Lcal}:=\|(I-\Lcal)^{s/2}f\|_{L^2}.$ 
At this moment, we note that the formal self-adjointness makes these Sobolev
spaces well defined in our setting, e.g. through the Plancherel formula on the Fourier
transform side as in \eqref{EQ:Hsub-norm}.

If $\Lcal$ is a Laplacian, e.g. if $\Lcal=X_1^2+\cdots+ X_n^2$ for
a basis of vector fields in $\mathfrak g$,
we will omit the 
subscript and simply write $H^s(G)$ in this case. Since the Laplacian
is elliptic, the spaces $H^s$ coincide with the usual Sobolev space on
$G$ considered as a smooth manifold. In Section \ref{SEC:Prelim}
we will analyse the main relevant properties of these spaces.

Let us now formulate the corresponding results. 
The first result deals with strictly positive and regular
propagation speed $a(t)$.

\begin{thm}[Case 1]
\label{theo_case_1}
Assume that $a\in C^1([0,T])$ and that $a(t)\ge a_0>0$.
For any $s\in\Rr$, if the Cauchy data satisfy
$(u_0,u_1)\in {H}^{1+s}_\Lcal \times {H}^{s}_\Lcal$,
then the Cauchy problem \eqref{CPa} has a unique solution 
$u\in C([0,T],{H}^{1+s}_\Lcal) \cap
C^1([0,T],{H}^{s}_\Lcal)$ which satisfies the estimate
\begin{equation}
\label{case_1_last-est}
\|u(t,\cdot)\|_{{H}^{1+s}_\Lcal}^2+\| \partial_t u(t,\cdot)\|_{{H}^s_\Lcal}^2\leq
C (\| u_0\|_{{H}^{1+s}_\Lcal}^2+\|u_1\|_{{H}^{s}_\Lcal}^2).
\end{equation}
Furthermore, suppose that the vector fields $X_1,\ldots,X_k$ satisfy H\"ormander condition of 
order $r$, i.e. that their iterated commutators of length $\leq r$ span the Lie algebra of $G$. 
Then the Cauchy problem \eqref{CPa} is well-posed in $C^\infty(G)$ and in $\Dcal'(G)$.
Moreover, for any $s\geq 0$, we have the estimate in the usual Sobolev spaces:
\begin{equation}
\label{case_1_last-SOB}
\|u(t,\cdot)\|_{{H}^{(1+s)/r}}^2+\| \partial_t u(t,\cdot)\|_{{H}^{s/r}}^2\leq
C (\| u_0\|_{{H}^{1+s}}^2+\|u_1\|_{{H}^{s}}^2).
\end{equation}
\end{thm}

We then deal with the situations when the function $a(t)$ may become zero or when it is less
regular than $C^1$. In this case, already for elliptic $\Lcal$ (for example, $\Lcal$ being the Laplacian),
we can not expect the well-posedness in $C^\infty$ on in $\Dcal'$, 
by an adaptation of results in \cite{Colombini-Spagnolo:Acta-ex-weakly-hyp} and
\cite{Colombini-Jannelli-Spagnolo:Annals-low-reg}.
However, it would hold in Gevrey spaces but the appearing Gevrey space may depend 
on the operator $\Lcal$. 

Thus, for $0<s<\infty$,
we define the $\Lcal$-Gevrey
space $\gamma^s_\Lcal(G)\subset C^\infty(G)$ by
\begin{equation}\label{DEF:GevL}
f\in \gamma^s_\Lcal(G) \Longleftrightarrow
\exists A>0: 
\| \esp^{A(-\Lcal)^{\frac{1}{2s}}} f\|_{L^2(G)} <\infty.
\end{equation}
The expression on the right is well-defined, for example in the sense of semi-groups,
since the operator $-\Lcal$ is formally self-adjoint and positive.
It can be also easily understood on the Fourier transform side, see \eqref{DEF:GevL-2}.
The first part of the following proposition justifies the terminology.
\begin{prop}\label{PROP:Gevrey}
We have the following properties.
\begin{itemize}
\item[(i)] If $\Lcal=X_1^2+\cdots+X_n^2$ 
is the Laplacian on $G$ then for $1\leq s<\infty$, the space $\gamma^s_\Lcal(G)$ in local coordinates coincides 
with the usual Gevrey space $\gamma^s(\Rn)$, i.e. the space of smooth functions $\psi\in C^\infty(\Rn)$
for which there exist constants $C>0$ and $A>0$ such that
$$
|\partial^\alpha\psi(x)|\leq C A^{|\alpha|} (\alpha!)^{s}.
$$
In the case of the Laplacian $\Lcal$, we denote the space $\gamma^s_\Lcal(G)$ simply
by $\gamma^s(G)$ dropping the subscript $\Lcal$.
\item[(ii)] If $\Lcal=X_1^2+\cdots+X_k^2$ with $X_1,\ldots,X_k$ satisfying the H\"ormander condition
of order $r$, i.e. their iterated commutations of length $\leq r$ span the Lie algebra of $G$, then
we have the continuous inclusions
$$
\gamma^s(G)\subset \gamma^s_\Lcal(G) \subset \gamma^{rs}(G).
$$
\end{itemize}
\end{prop}
We note that Part (i) of Proposition \ref{PROP:Gevrey} has been proved in
\cite{Dasgupta-Ruzhansky:Gevrey-arxiv}. The continuous embeddings in Part (ii)
follow from the formula \eqref{DEF:GevL-2} and estimates
\eqref{EQ:nu-ests} in Proposition \ref{PROP:subL}. We also note that especially 
for $s\geq 1$, dropping the subscript $\Lcal$ in the notation in the case of 
Laplacians should not cause problems since in this case the space coincides
with the usual Gevrey space on manifolds, as stated in 
Part (i) of Proposition \ref{PROP:Gevrey}.

We formulate the well-posedness of the Cauchy problem \eqref{CPa} in the Gevrey spaces
$\gamma^s_\Lcal$. If the H\"ormander condition is satisfied, the embeddings
in Part (ii) of Proposition \ref{PROP:Gevrey}, in view of Part (i) of Proposition \ref{PROP:Gevrey},
yield a well-posedness results in the Gevrey spaces $\gamma^s$, or in the usual $\gamma^s(\Rn)$ in
local coordinates, provided all the Gevrey indices are $\geq 1$. 
However, the well-posedness formulated in $\gamma^s_\Lcal$ is a more refined
statement since the space $\gamma^s_\Lcal$ is in general bigger than $\gamma^s$,
or maybe unrelated to it if the H\"ormander condition is not satisfied.

\begin{thm}[Case 2]
\label{theo_case_2}
Assume that $a(t)\ge a_0>0$ and that $a\in C^\alpha([0,T])$ with $0<\alpha<1$. Then
for initial data $u_0,u_1\in \gamma^s_\Lcal(G)$, the Cauchy problem \eqref{CPa} has 
a unique solution $u\in C^2([0,T], \gamma^s_\Lcal(G))$, provided that
\begin{equation}\label{EQ:Case2-s}
1\le s<1+\frac{\alpha}{1-\alpha}.
\end{equation}
Furthermore, suppose that the vector fields $X_1,\ldots,X_k$ satisfy H\"ormander condition of 
order $r$, i.e. that their iterated commutators of length $\leq r$ span the Lie algebra of $G$. 
Then, in particular, for initial data $u_0,u_1\in \gamma^s(G)$ with $s$ satisfying
\eqref{EQ:Case2-s}, the Cauchy problem \eqref{CPa} has 
a unique solution $u\in C^2([0,T], \gamma^{rs}(G))$.
\end{thm}
The second part of Theorem \ref{theo_case_2} follows from the first one in view of
the embeddings in Proposition \ref{PROP:Gevrey}, Part (ii).
By Proposition \ref{PROP:Gevrey}, Part (i), for $s\geq 1$,
the spaces $\gamma^{s}(G)$ and $\gamma^{rs}(G)$ can be identified
with Gevrey spaces $\gamma^{s}(\Rn)$ and $\gamma^{rs}(\Rn)$ in local coordinates,
repectively.
Consequently, we obtain the local version of Theorem \ref{theo_case_2} with loss of
Gevrey regularity in local coordinates:

\begin{cor}\label{COR:case-2}
Assume that the vector fields $X_1,\ldots,X_k$ satisfy H\"ormander condition of 
order $r$. Assume further that 
$a(t)\ge a_0>0$ and that $a\in C^\alpha([0,T])$ with $$0<\alpha<1.$$
Let the initial data $u_0,u_1$ belong to $\gamma^s(\Rn)$ in any local coordinate chart,
for
$$
1 \leq s<1+\frac{\alpha}{1-\alpha}.
$$
Then the Cauchy problem \eqref{CPa} has 
a unique solution $u$ such that $u(t,\cdot)$ belongs to $\gamma^{rs}(\Rn)$
in every local coordinate chart.
\end{cor}
We now consider the situation when the propagation speed $a(t)$ may become
zero but is regular, i.e. $a\in C^\ell$ for $\ell\geq 2$.

\begin{thm}[Case 3]
\label{theo_case_3}
Assume that $a(t)\ge 0$ and that $a\in C^\ell([0,T])$ with $\ell\ge 2$. 
Then for initial data $u_0,u_1\in \gamma^s_\Lcal(G)$, the Cauchy problem \eqref{CPa} has a unique solution 
$u\in C^2([0,T], \gamma^s_\Lcal(G))$, provided that
\begin{equation}\label{EQ:s1}
1\le s<1+\frac{\ell}{2}.
\end{equation}
If $a(t)\ge 0$ belongs to $C^\infty([0,T])$ then the Cauchy problem \eqref{CPa} is well-posed in every 
Gevrey class $\gamma^s_\Lcal(G)$, $s\geq 1$.
\end{thm}
We now consider the case which is complementary to that in Theorem \ref{theo_case_3},
namely, when the propagation speed $a(t)$ may become
zero and is less regular, i.e. $a\in C^\alpha$ for $0<\alpha<2$.

\begin{thm}[Case 4]
\label{theo_case_4}
Assume that $a(t)\ge 0$ and that $a\in C^\alpha([0,T])$ with $0<\alpha<2$. 
Then, for initial data $u_0,u_1\in \gamma^s_\Lcal(G)$ the Cauchy problem \eqref{CPa} has a unique solution 
$u\in C^2([0,T], \gamma^s_\Lcal(G))$, provided that
\begin{equation}\label{EQ:s2}
1\le s<1+\frac{\alpha}{2}.
\end{equation}
\end{thm}

Theorems \ref{theo_case_3} and \ref{theo_case_4} have obvious consequences,
similar to those in the second part of Theorem \ref{theo_case_2} and in 
Corollary \ref{COR:case-2}. 
Namely, for initial data $u_0,u_1\in \gamma^s(G)$, $s\geq 1$,
the Cauchy problem \eqref{CPa} has a unique solution 
$u\in C^2([0,T], \gamma^{rs}(G))$, 
provided that $s$ also satisfies conditions \eqref{EQ:s1} or
\eqref{EQ:s2}, respectively.

We note that we could have united formulations
of Theorems \ref{theo_case_3} and \ref{theo_case_4} 
in a single statement but we decided to separate them since
our proofs of these two theorems are in fact very different, based on
quasi-symmetrisers and regularisation and separation of 
characteristic roots, respectively.

Finally, we note that using a characterisation of ultradistributions on compact
Lie groups from \cite{Dasgupta-Ruzhansky:Gevrey-arxiv}, one can obtain
counterparts of the Gevrey results also in the corresponding spaces of
ultradistributions (see \cite{GR:11} for an example of such an argument in $\Rn$).

\section{Fourier analysis and symbolic properties of sub-Laplacians}
\label{SEC:Prelim}

In this section we recall the necessary elements of the global Fourier analysis
that we will be using, and establish properties of the matrix symbols of
sub-Laplacians, leading to embedding properties of the associated Sobolev spaces. 
The matrix symbols for operators on compact Lie groups have
been developed in \cite{Ruzhansky-Turunen:BOOK,Ruzhansky-Turunen:IMRN}
to which we refer also for the details of the Fourier analysis reviewed below.

Let $\Gh$ be the unitary dual of $G$, consisting of the
equivalence classes $[\xi]$ of the continuous irreducible unitary representations
$\xi:G\to \C^{\dxi\times \dxi}$,
of matrix-valued functions satisfying
$\xi(xy)=\xi(x)\xi(y)$ and $\xi(x)^{*}=\xi(x)^{-1}$ for all
$x,y\in G$.
For a function
$f\in C^{\infty}(G)$ we can define its Fourier coefficient at $[\xi]\in\Gh$ by
$$
\widehat{f}(\xi):=\int_{G} f(x) \xi(x)^{*}dx\in\Cdxi,
$$
where the integral is (always) taken with respect to the Haar measure on $G$,
and with a natural extension to distributions.
The Fourier series becomes
$$
f(x)=\sum_{[\xi]\in\Gh} \dxi \Tr\p{\xi(x)\widehat{f}(\xi)},
$$
with the Plancherel's identity taking the form
\begin{equation}\label{EQ:Plancherel}
\|f\|_{L^{2}(G)}=\p{\sumxi \dxi
\|\widehat{f}(\xi)\|_{\HS}^{2}}^{1/2}=:
\|\widehat{f}\|_{\ell^{2}(\Gh)},
\end{equation}
which we take as the definition of the norm on the Hilbert space
$\ell^{2}(\Gh)$, and where 
$ \|\widehat{f}(\xi)\|_{\HS}^{2}=\Tr(\widehat{f}(\xi)\widehat{f}(\xi)^{*})$ is the
Hilbert--Schmidt norm of the matrix $\widehat{f}(\xi)$.
For a Laplacian $\Delta$ on $G$, we have that for a fixed $[\xi]\in\Gh$,
all $\xi_{ij}(x)$, $1\leq i,j\leq \dxi$, are eigenfunctions of $-\Delta$ with the
same eigenvalue, which we denote by $|\xi|^2$, so that we have
$$
-\Delta \xi_{ij}(x)=|\xi|^2 \xi_{ij}(x) \textrm{ for all } 1\leq i,j\leq \dxi.
$$
We denote $\jp{\xi}:=(1+|\xi|^2)^{1/2}$, which are the eigenvalues of the 
first order elliptic operator $(I-\Delta)^{1/2}.$

Smooth functions and distributions on $G$ can be characterised in terms of
their Fourier coefficients. Thus, we have 
$$
f\in C^{\infty}(G)\Longleftrightarrow
\forall N \;\exists C_{N} \textrm{ such that }
\|\widehat{f}(\xi)\|_{\HS}\leq C_{N} \jp{\xi}^{-N}
\textrm{ for all } [\xi]\in\Gh.
$$ 
Also, for distributions, we have
$$
u\in \Dcal'(G)
\Longleftrightarrow
\exists M \;\exists C \textrm{ such that }
\|\widehat{u}(\xi)\|_{\HS}\leq C\jp{\xi}^{M}
\textrm{ for all } [\xi]\in\Gh.
$$
Furthermore, importantly for our results,  it was established 
in \cite{Dasgupta-Ruzhansky:Gevrey-arxiv} that
the Gevrey ultradifferentiable functions and ultradistributions on 
compact Lie groups, initially defined in local coordinates, can be
also characterised in terms of their Fourier coefficients. Thus, for $s\geq 1$,
$$
f\in \gamma^{s}(G)\Longleftrightarrow
\exists A>0, \; C>0 \textrm{ such that }
\|\widehat{f}(\xi)\|_{\HS}\leq C \esp^{-A\jp{\xi}^{1/s}}
\textrm{ for all } [\xi]\in\Gh.
$$ 
Here the space $\gamma^{s}(G)$ is the usual Gevrey space $\gamma^s(\Rn)$
extended to $G$ viewed as an analytic manifold, as in
Proposition \ref{PROP:Gevrey}, Part (i).

Given a linear continuous operator $T:C^{\infty}(G)\to C^{\infty}(G)$
(or even $T:C^{\infty}(G)\to \Dcal'(G)$), we define its matrix symbol by
$\sigma_{T}(x,\xi):=\xi(x)^{*} (T \xi)(x)\in \Cdxi$, where
$T \xi$ means that we apply $T$ to the matrix components of $\xi(x)$.
In this case we may prove that 
\begin{equation}\label{EQ:T-op}
Tf(x)=\sumxi \dxi \Tr\p{\xi(x)\sigma_{T}(x,\xi)\widehat{f}(\xi)}.
\end{equation}
The correspondence between operators and symbols is one-to-one, and
we will write $T_{\sigma}$ for the operator given by 
\eqref{EQ:T-op} corresponding to the symbol $\sigma(x,\xi)$.
The quantization \eqref{EQ:T-op} has been extensively studied in
\cite{Ruzhansky-Turunen:BOOK,Ruzhansky-Turunen:IMRN}, to which we
refer for its properties and for the corresponding symbolic calculus.

In particular, if $X_1,\ldots,X_n$ is an orthonormal basis of the Lie algebra of $G$,
then the symbol of the Laplacian 
$\Delta=X_1^2+\cdots+X_n^2$ is
$\sigma_\Delta(\xi)=-|\xi|^2 I_\dxi$, where 
$I_\dxi\in\Cdxi$ is the identity matrix.

We now turn to analysing properties of the matrix symbol of the operator
\eqref{EQ:subL}, namely, of the operator
$$
\Lcal=X_1^2+\cdots+X_k^2.
$$
The operator $\Lcal$ is formally self-adjoint, therefore its symbol 
$\sigma_\Lcal$ can be diagonalised by a choice of the basis in the representation
spaces. Moreover, the operator $-\Lcal$ is positive definite as sum of squares of
vector fields. Therefore, without loss of generality, we can always write
\begin{equation}\label{EQ:subL-symbol}
\sigma_{-\Lcal}(\xi)=
\left(\begin{matrix}
\nu_1^2(\xi) &  0  & \ldots & 0\\
0  & \nu_2^2(\xi)  & \ldots & 0\\
\vdots & \vdots & \ddots & \vdots\\
0  &   0       &\ldots & \nu_\dxi^2(\xi)
\end{matrix}\right),
\end{equation}
for some $\nu_j(\xi)\geq 0$.

Consequently, we can also define Sobolev spaces $H^s_\Lcal(G)$ associated to 
sums of squares. Thus, for any $s\in\Rr$, we set
\begin{equation}\label{EQ:HsL}
H^s_\Lcal(G):=\left\{ f\in\Dcal'(G): (I-\Lcal)^{s/2}f\in L^2(G)\right\},
\end{equation}
with the norm $\|f\|_{H^s_\Lcal}:=\|(I-\Lcal)^{s/2}f\|_{L^2}.$ Using Plancherel's
identity \eqref{EQ:Plancherel}, we can write
\begin{multline}\label{EQ:Hsub-norm}
\|f\|_{H^s_\Lcal}=\|(I-\Lcal)^{s/2}f\|_{L^2}=
\p{\sumxi \dxi \|(I_\dxi-\sigma_\Lcal(\xi))^{s/2}\widehat{f}(\xi)\|_\HS^2}^{1/2} \\
= \p{ \sumxi \dxi \sum_{j=1}^\dxi (1+\nu_j^2(\xi))^{s} \sum_{m=1}^\dxi
|\widehat{f}(\xi)_{jm}|^2  }^{1/2}.
\end{multline}
There are different characterisations of such Sobolev spaces, also in more generality:
for example, see \cite{FMV} for a heat kernel description, etc.
However, for our purposes, the Fourier description \eqref{EQ:Hsub-norm} will suffice.

We note that using the Plancherel identity,
the Gevrey space $\gamma^s_\Lcal(G)$ defined in \eqref{DEF:GevL}
can be characterised by the condition

\begin{multline}\label{DEF:GevL-2}
f\in \gamma^s_\Lcal \Longleftrightarrow
\exists A>0: 
\| \esp^{A(-\Lcal)^{\frac{1}{2s}}} f\|_{L^2}^2 = \sumxi \dxi
\|\esp^{A\sigma_{-\Lcal}(\xi)^{\frac{1}{2s}}}\widehat{f}(\xi)\|_{\HS}^2 \\
= \sumxi \dxi \sum_{j=1}^\dxi 
\esp^{A\nu_j(\xi)^{1/s}} \sum_{m=1}^\dxi
|\widehat{f}(\xi)_{jm}|^2  
<\infty,
\end{multline}
where now the matrix $e^{A\sigma_{-\Lcal}(\xi)^{\frac{1}{2s}}}$ is well-defined in view of the
diagonal form of $\sigma_{-\Lcal}(\xi)$ in \eqref{EQ:subL-symbol}.

We now assume that $X_1,\ldots,X_k$ is a family of left-invariant vector fields such that 
their iterated commutators of length $\leq r$ span the Lie algebra of $G$, and 
establish a relation between $\nu_j(\xi)$ and the eigenvalues of the Laplacian,
yielding also the embedding properties between Sobolev spaces $H^s_\Lcal(G)$ and
the usual Sobolev spaces $H^s(G)$ on $G$ viewed as a smooth manifold. 
Here we note that the usual Sobolev spaces $H^s=H^s(G)$ can be also characterised
as the set of $f\in \Dcal'(G)$ such that $(I-\Delta)^{s/2}f\in L^2(G)$, with the
corresponding equivalence of norms. For integers $s\in\N$, the embedding
$H^s_\Lcal\subset H^{s/r}$ in \eqref{EQ:Hsub-incl} is, in fact, Theorem 13 in 
\cite{Rothschild-Stein:AM-1976}.

\begin{prop}\label{PROP:subL}
Let $X_1,\ldots,X_k$ is a family of left-invariant vector fields such that 
their iterated commutators of length $\leq r$ span the Lie algebra of $G$.
Let 
$$
\Lcal:=X_1^2+\cdots+X_k^2
$$
be the corresponding sub-Laplacian with symbol \eqref{EQ:subL-symbol}.
Then there exists a constant $c>0$ such that 
\begin{equation}\label{EQ:nu-ests}
c\jp{\xi}^{1/r}\leq \nu_j(\xi)+1\leq \sqrt{2} \jp{\xi} \textrm{ for all }
[\xi]\in\Gh \textrm{ and } 1\leq j\leq\dxi.
\end{equation}
Consequently, for $s\geq 0$ we have the continuous embeddings
\begin{equation}\label{EQ:Hsub-incl}
H^s\subset H^s_\Lcal\subset H^{s/r}
\textrm{ and }
H^{-s/r}\subset H^{-s}_\Lcal\subset H^{-s}.
\end{equation}
More precisely, for any $s\geq 0$ there exist constants 
$C_1, C_2>0$ such that we have
\begin{equation}\label{EQ:Hsub-embs}
C_1\|f\|_{H^{s/r}}\leq \|f\|_{H^s_\Lcal}\leq C_2 \|f\|_{H^s}
\textrm{ and }
C_1\|f\|_{H^{-s}}\leq \|f\|_{H^{-s}_\Lcal}\leq C_2 \|f\|_{H^{-s/r}}.
\end{equation}
\end{prop}
\begin{proof}
The proof of \eqref{EQ:nu-ests} is easy if we use the following result
by Rothschild and Stein \cite{Rothschild-Stein:AM-1976}.
In Theorem 18 in \cite{Rothschild-Stein:AM-1976} it was
shown, in particular, that for a sub-Laplacian $\Lcal$ satisfying 
H\"ormander condition of order $\leq r$, we have the estimate
$$
\|f\|^2_{H^{2/r}}\leq C(\|\Lcal f\|^2_{L^2}+\|f\|^2_{L^2}).
$$
Using the Fourier series and Plancherel's theorem, this is equivalent to
the estimate
$$
\sumxi \dxi \jp{\xi}^{4/r} \|\widehat{f}(\xi)\|_\HS^2 \leq 
C\sumxi \dxi (\|\sigma_\Lcal(\xi) \widehat{f}(\xi)\|_\HS^2
+ \|\widehat{f}(\xi)\|_\HS^2),
$$
holding for all $f\in H^2_\Lcal$. 
In particular, applying this to $f$ such that $\widehat{f}(\xi)=A$ for some
$[\xi]\in\Gh$ and zero for all other $[\xi]$, it follows that we have
the estimate
$$
\jp{\xi}^{2/r} \|A\|_\HS\leq C(\|\sigma_\Lcal(\xi) A\|_\HS+\|A\|_\HS)
$$
for all $A\in\Cdxi$. Now, recalling that the symbol $\sigma_\Lcal$ is
diagonal of the form \eqref{EQ:subL-symbol}, we obtain that
$\jp{\xi}^{2/r} \leq C(\nu_j^2(\xi)+1)$ for all $1\leq j\leq\dxi$,
proving the first (left) inequality in \eqref{EQ:nu-ests}.
The second estimate  in \eqref{EQ:nu-ests} follows 
from the relation $\jp{\xi}=(1+|\xi|^2)^{1/2}$ and the
estimate $\nu_j^2(\xi)\leq |\xi|^2$, which is a consequence of
the fact that the operator $\Delta-\Lcal=X_{k+1}^2+\cdots +X_n^2$
is formally self-adjoint and negative definite.

To obtain \eqref{EQ:Hsub-embs}, we observe that the second estimate
in \eqref{EQ:Hsub-embs} follows from the first one by duality.
In turn,  the first part of \eqref{EQ:Hsub-embs} follows from
\eqref{EQ:Hsub-norm} using estimate 
\eqref{EQ:nu-ests}.
\end{proof}

\begin{rem}\label{REM:hyp}
We note that the Cauchy problem \eqref{CPa} in local coordinates
is always hyperbolic.
In fact, the positivity of the matrix symbol $\sigma_{-\Lcal}$
implies that the operator $\Lcal$ satisfies the sharp
G{\aa}rding inequality, see \cite{Ruzhansky-Turunen:JFA-Garding}.
Consequently, the principal symbol of $-\Lcal$ in any local
coordinate system is non-negative, implying that the operator
$\partial_t^2-\Lcal$ is hyperbolic.
\end{rem}

\section{Reduction to first order system and energy estimates}
\label{SEC:reduction}

The operator $\Lcal$  has the symbol \eqref{EQ:subL-symbol}, which we can write
in matrix components as 
$$
\sigma_{-\Lcal}(\xi)_{mk}=\nu_m^2(\xi)\delta_{mk},
\; 1\leq m,k\leq \dxi,
$$ 
with $\delta_{mk}$ standing for the Kronecker's delta.
Taking the Fourier transform of
\eqref{CPa}, we obtain the collection of Cauchy problems for 
matrix-valued Fourier coefficients:
\begin{equation}\label{CPa-FC}
\partial_{t}^{2}\widehat{u}(t,\xi)-a(t)\sigma_{\Lcal}(\xi)\widehat{u}(t,\xi)=0,
\; [\xi]\in\Gh.
\end{equation}
Writing this in the matrix form, we see that this is equivalent to the system
$$
\partial_{t}^{2} \widehat{u}(t,\xi)+
a(t) \left(\begin{matrix}
\nu_1^2(\xi) &  0  & \ldots & 0\\
0  &  \nu_2^2(\xi) & \ldots & 0\\
\vdots & \vdots & \ddots & \vdots\\
0  &   0       &\ldots & \nu_\dxi^2(\xi)
\end{matrix}\right) \widehat{u}(t,\xi)=0,
$$
where we put explicitly the diagonal symbol 
$\sigma_{\Lcal}(\xi)$.
Rewriting \eqref{CPa-FC} in terms of matrix coefficients
$\widehat{u}(t,\xi)=\left(\widehat{u}(t,\xi)_{mk}\right)_{1\leq m,k\leq \dxi}$, 
we get the equations
\begin{equation}\label{EQ:WE-scalars}
\partial_{t}^{2} \widehat{u}(t,\xi)_{mk}+ a(t) \nu_m^2(\xi)
 \widehat{u}(t,\xi)_{mk}=0,\qquad [\xi]\in\Gh,\;
1\leq m,k\leq \dxi.
\end{equation}
The main point of our further analysis is that we can make an individual
treatment of the equations in \eqref{EQ:WE-scalars}.
Thus, let us fix $[\xi]\in\Gh$ and $m,k$ with $1\leq m,k\leq \dxi$,
and let us denote 
$\widehat{v}(t,\xi):=\widehat{u}(t,\xi)_{mk}$. 
We then study the Cauchy problem
\begin{equation}\label{EQ:WE-v}
\partial_{t}^{2} \widehat{v}(t,\xi)+ a(t) \nu_m^2(\xi)
 \widehat{v}(t,\xi)=0,\quad 
 \widehat{v}(t,\xi)=\widehat{v}_{0}(\xi), \; 
 \partial_{t}\widehat{v}(t,\xi)=\widehat{v}_{1}(\xi),
\end{equation}
with $\xi,m$ being parameters, and want to derive estimates
for $\widehat{v}(t,\xi)$. Combined with characterisations of 
Sobolev, smooth and Gevrey functions, this will yield the well-posedness
results for the original Cauchy problem \eqref{CPa}.

In the sequel, for fixed $m$, we set
\begin{equation}\label{xi_l}
|\xi|_\nu:=\nu_m^2(\xi).
\end{equation}
Hence, the equation in \eqref{EQ:WE-v} can be written as
\begin{equation}
\label{eq_xi}
\partial_{t}^{2} \widehat{v}(t,\xi)+a(t)|\xi|_\nu^2\widehat{v}(t,\xi)=0.
\end{equation}
Note that if $|\xi|_\nu\not=0$,
the equation \eqref{eq_xi} is of strictly hyperbolic type
in Case 1 and 2 and weakly hyperbolic in Case 3 and 4.  We now proceed 
with a standard reduction to a first order system of this equation and define the corresponding energy. 
The energy estimates will be given in terms of $t$ and $|\xi|_\nu$ and we then go back to $t$, 
$\xi$ and $m$ by using \eqref{xi_l}.

We now use the transformation
\[
\begin{split}
V_1&:=i|\xi|_\nu\widehat{v},\\
V_2&:= \partial_t \widehat{v}.
\end{split}
\]
It follows that the equation \eqref{eq_xi} can be written as the first order system
\begin{equation}\label{EQ:system}
\partial_t V(t,\xi)=i|\xi|_\nu A(t)V(t,\xi),
\end{equation}
where $V$ is the column vector with entries $V_1$ and $V_2$ and
\[
A(t)=\left(
    \begin{array}{cc}
      0 & 1\\
      a(t) & 0 \\
           \end{array}
  \right).
\]
The initial conditions $\widehat{v}(0,\xi)=\widehat{v}_{0}(\xi)$, $\partial_{t}\widehat{v}(0,\xi)=\widehat{v}_{1}(\xi)$ 
are transformed into
\[
V(0,\xi)=\left(
    \begin{array}{c}
      i|\xi|_\nu \widehat{v}_0(\xi)\\
      \widehat{v}_{1}(\xi)
     \end{array}
  \right).
\]
Note that the matrix $A$ has eigenvalues $\pm\sqrt{a(t)}$ and symmetriser
\[
S(t)=\left(
    \begin{array}{cc}
      2a(t) & 0\\
      0 & 2 \\
           \end{array}
  \right).
\]
By definition of the symmetriser we have that 
\[
SA-A^\ast S=0.
\]
It is immediate to prove that
\begin{equation}
\label{est_sym}
2\min_{t\in[0,T]}(a(t),1)|V|^2\le (SV,V)\le 2\max_{t\in[0,T]}(a(t),1)|V|^2,
\end{equation}
where $(\cdot,\cdot)$ and $|\cdot|$ denote the inner product and the norm in $\C^2$, respectively.

\subsection{Case 1: Proof of Theorem \ref{theo_case_1}}

In Case 1 ($a(t)>0$, $a\in C^1([0,T])$) it is clear that there exist constants $a_0>0$ and $a_1>0$ such that 
\[
a_0=\min_{t\in[0,T]}a(t)
\; \textrm{ and } \;
a_1=\max_{t\in[0,T]}{a(t)}.
\]
Hence \eqref{est_sym} implies,
\begin{equation}
\label{est_sym_1}
c_0|V|^2=2\min(a_0,1)|V|^2\le (SV,V)\le 2\max(a_1,1)|V|^2=c_1|V|^2,
\end{equation}
with $c_0,c_1>0$.
We then define the energy 
$$E(t,\xi):=(S(t)V(t,\xi),V(t,\xi)).$$ 
We get, from \eqref{est_sym_1}, that
\begin{multline*}
\partial_t E(t,\xi)=(\partial_t S(t)V(t,\xi),V(t,\xi))+(S(t)\partial_t V(t,\xi),V(t,\xi))+(S(t)V(t,\xi),\partial_t V(t,\xi))\\
=(\partial_t S(t)V(t,\xi),V(t,\xi))+i|\xi|_\nu (S(t)A(t)V(t,\xi),V(t,\xi))-i|\xi|_\nu (S(t)V(t,\xi),A(t)V(t,\xi))\\
=(\partial_t S(t)V(t,\xi),V(t,\xi))+i|\xi|_\nu ((SA-A^\ast S)(t)V(t,\xi),V(t,\xi))\\
=(\partial_t S(t)V(t,\xi),V(t,\xi))\le \Vert \partial_t S\Vert |V(t,\xi)|^2\le c' E(t,\xi)
\end{multline*}
i.e. we obtain
\begin{equation}
\label{E_1}
\partial_t E(t,\xi)\le c' E(t,\xi),
\end{equation}
for some constant $c'>0$. By Gronwall's lemma applied to inequality
\eqref{E_1} we conclude that for all $T>0$ there exists $c>0$ such that
\[
E(t,\xi)\le c E(0,\xi).
\]
Hence, inequalities \eqref{est_sym_1} yield
\[
c_0|V(t,\xi)|^2\le E(t,\xi)\le c E(0,\xi)\le cc_1|V(0,\xi)|^2,
\]
for constants independent of $t\in[0,T]$ and $\xi$. This allows us to write the following statement: 
there exists a constant $C_1>0$ such that
\begin{equation}
\label{case_1_est}
|V(t,\xi)|\le C_1 |V(0,\xi)|,
\end{equation}
for all $t\in[0,T]$ and $\xi$. Hence
\[
|\xi|_\nu^2 |\widehat{v}(t,\xi)|^2+|\partial_t \widehat{v}(t,\xi)|^2
\le C_1'( |\xi|_\nu^2 |\widehat{v}_0(\xi)|^2+|\widehat{v}_1(\xi)|^2).
\]
Recalling the notation 
$\widehat{v}(t,\xi)=\widehat{u}(t,\xi)_{mk}$ and $|\xi|_\nu=\nu_m(\xi)$, this means
\begin{equation}
\label{case_1_est_mn}
\nu_m^2(\xi) |\widehat{u}(t,\xi)_{mk}|^2+|\partial_t \widehat{u}(t,\xi)_{mk}|^2
\le C_1'( \nu_m^2(\xi) |\widehat{u}_0(\xi)_{mk}|^2+|\widehat{u}_1(\xi)_{mk}|^2)
\end{equation}
for all $t\in[0,T]$, $[\xi]\in\Gh$ and $1\le m,k\le\dxi$, with the
constant $C_1'$ independent of $\xi$, $m,k$.
Now we recall that by Plancherel's equality, we have
$$
\|\partial_t u(t,\cdot)\|_{L^2}^2=\sumxi \dxi \|\partial_t \widehat{u}(t,\xi)\|_\HS^2=
\sumxi \dxi \sum_{m,k=1}^\dxi |\partial_t \widehat{u}(t,\xi)_{mk}|^2
$$
and 
$$
\|\Lcal u(t,\cdot)\|_{L^2}^2=\sumxi \dxi   \| \sigma_\Lcal(\xi) \widehat{u}(t,\xi)\|_\HS^2=
\sumxi \dxi \sum_{m,k=1}^\dxi  \nu_m^2(\xi) |\widehat{u}(t,\xi)_{mk}|^2.
$$
Hence, the estimate \eqref{case_1_est_mn} implies that
\begin{equation}
\label{case_1_last}
\|\Lcal u(t,\cdot)\|_{L^2}^2+\|\partial_t u(t,\cdot)\|_{L^2}^2\leq
C (\|\Lcal u_0\|_{L^2}^2+\|u_1\|_{L^2}^2),
\end{equation}
where the constant $C>0$ does not depend on $t\in[0,T]$. 
More generally, modulo analytic functions corresponding to trivial representations,
multiplying \eqref{case_1_est_mn}
by powers of $\nu_m(\xi)$, for any $s$, we get
\begin{multline}
\label{case_1_est_mn2}
\nu_m^{2+2s}(\xi) |\widehat{u}(t,\xi)_{mk}|^2+\nu_m^{2s}(\xi)  |\partial_t 
\widehat{u}(t,\xi)_{mk}|^2 \\
\le C_1'( \nu_m^{2+2s}(\xi) |\widehat{u}_0(\xi)_{mk}|^2+\nu_m^{2s}(\xi)|\widehat{u}_1(\xi)_{mk}|^2).
\end{multline}
Taking the sum over $\xi$, $m$ and $k$ as above, this yields
the estimate \eqref{case_1_last-est}.

If the vector fields $X_1,\ldots, X_k$ satisfy H\"ormander's condition of order $r$, the estimate
\eqref{case_1_last-SOB} follows from \eqref{case_1_last-est} and Proposition \ref{PROP:subL}.
Consequently, taking $\alpha$ arbitrarily large, we can also conclude that the solution 
$u$ belongs to $C^\infty(G)$ and by duality to $\D'(G)$ in $x$ if the initial data belong to 
$C^\infty(G)$ and $\D'(G)$, respectively. This completes the proof of Theorem \ref{theo_case_1}.

\medskip
Before proceeding to proving Cases 1--3,
we note that in $\Rn$, due to the necessity to introduce compactly supported cut-offs to explore
the finite propagation speed of the equation, one has to distinguish between 
the analytic case $s=1$ and Gevrey cases $s>1$. The case $s=1$ can be then handled by
using, e.g. Kajitani \cite{Kajitani:analytic-CPDE-1986}, see
also earlier results by Bony and Shapira \cite{Bony-Shapira:analytic-IM-1972}.
However, with our method of proof, it will not be necessary to make such a distinction
since the group $G$ is already compact.

\subsection{Case 2: Proof of Theorem \ref{theo_case_2}}
We assume now still $a(t)\ge a_0>0$ but this time the regularity of $a$ is reduced, i.e.,  $a\in C^\alpha([0,T])$, 
with $0<\alpha<1$. As above $a(t)\ge a_0>0$ for all $t\in[0,T]$. Keeping the notation \eqref{xi_l} and
inspired by \cite{GR:11} we look for a solution of the system \eqref{EQ:system}, i.e. of
\begin{equation}
\label{system_A}
\partial_t V(t,\xi)={\rm i}|\xi|_\nu A(t)V(t,\xi),
\end{equation}
of the following form
\[
V(t,\xi)={\rm e}^{-\rho(t)|\xi|_\nu^{1/s}}(\det H)^{-1}HW,
\]
where 
$\rho\in C^1([0,T])$ is a real-valued function which will be suitably chosen in the sequel, 
$W=W(t,\xi)$ is to be determined,
\[
H(t)=\left(
    \begin{array}{cc}
      1 & 1 \\
      \lambda_1(t) & \lambda_2(t)  
         \end{array}
  \right),
\]
and, for $\varphi\in C^\infty_c(\mathbb{R})$, $\varphi\ge 0$ with integral $1$,  
\begin{equation}
\label{def_lambdaj}
\begin{split}
\lambda_1(t)&=(-\sqrt{a}\ast\varphi_\epsilon)(t),\\
\lambda_2(t)&=(+\sqrt{a}\ast\varphi_\epsilon)(t),
\end{split}
\end{equation}
$\varphi_\epsilon(t)=\frac{1}{\epsilon}\varphi(t/\epsilon).$
By construction, $\lambda_1$ and $\lambda_2$ (where the dependence on $\epsilon$ is omitted for the sake of simplicity) are smooth in $t\in[0,T]$. Moreover,
\[
\lambda_2(t)-\lambda_1(t)\ge 2\sqrt{a_0},
\]
for all $t\in[0,T]$ and $\eps\in(0,1]$,
\[
|\lambda_1(t)+\sqrt{a(t)}|\le c_1\eps^\alpha
\]
and
\[
|\lambda_2(t)-\sqrt{a(t)}|\le c_2\eps^\alpha,
\]
uniformly in $t$ and $\eps$.
By substitution in \eqref{system_A} we get
\begin{multline*}
\esp^{-\rho(t)|\xi|_\nu^{\frac{1}{s}}}(\det H)^{-1}H\partial_tW+\esp^{-\rho(t)
|\xi|_\nu^{\frac{1}{s}}}(-\rho'(t)|\xi|_\nu^{\frac{1}{s}})(\det H)^{-1}HW-
\esp^{-\rho(t)|\xi|_\nu^{\frac{1}{s}}}\frac{\partial_t\det H}{(\det H)^2}HW\\
+\esp^{-\rho(t)|\xi|_\nu^{\frac{1}{s}}}(\det H)^{-1}(\partial_tH)W
={\rm i}|\xi|_\nu \esp^{-\rho(t)|\xi|_\nu^{\frac{1}{s}}}(\det H)^{-1}AHW.
\end{multline*}
Multiplying both sides of the previous equation by $\esp^{\rho(t)|\xi|_\nu^{\frac{1}{s}}}(\det H)H^{-1}$ we get
\[
\partial_tW-\rho'(t)|\xi|_\nu^{\frac{1}{s}}W-\frac{\partial_t\det H}{\det H}W + H^{-1}(\partial_t H)W= {\rm i}|\xi|_\nu H^{-1}AHW.
\]
Hence,
\begin{multline}
\label{energy}
\partial_t |W(t,\xi)|^2=2{\rm Re} (\partial_t W(t,\xi),W(t,\xi))\\
=2\rho'(t)|\xi|_\nu^{\frac{1}{s}}|W(t,\xi)|^2+2\frac{\partial_t\det H}{\det H}|W(t,\xi)|^2-2{\rm Re}(H^{-1}\partial_t HW,W)\\
-2|\xi|_\nu {\rm Im} (H^{-1}AHW,W).
\end{multline}
It follows that
\begin{multline}
\label{energy_case2}
\partial_t |W(t,\xi)|^2\le 2{\rm Re} (\partial_t W(t,\xi),W(t,\xi))\le 
2\rho'(t)|\xi|_\nu^{\frac{1}{s}}|W(t,\xi)|^2+2\biggl |\frac{\partial_t\det H}{\det H}\biggr||W(t,\xi)|^2\\
+2\Vert H^{-1}\partial_t H\Vert|W(t,\xi)^2|+|\xi|_\nu \Vert H^{-1}AH-(H^{-1}AH)^\ast\Vert |W(t,\xi)|^2.
\end{multline}

We proceed by estimating
\begin{enumerate}
\item $\frac{\partial_t\det H}{\det H}$,
\item $\Vert H^{-1}\partial_t H\Vert$,
\item $\Vert H^{-1}AH-(H^{-1}AH)^\ast\Vert$.
\end{enumerate} 

Note that the matrices $H$ and $A$ depend only on $t$ here. Hence, in complete analogy with \cite{GR:11} (Remark 21) we get that for all $T>0$ there exist constants $c_1,c_2,c_3>0$ such that
\begin{equation}
\label{case_2_1}
\biggl |\frac{\partial_t\det H}{\det H}\biggr |\le c_1\eps^{\alpha-1},
\end{equation}
\begin{equation}
\label{case_2_2}
\Vert H^{-1}\partial_t H\Vert\le c_2\eps^{\alpha-1},
\end{equation}
\begin{equation}
\label{case_2_3}
\Vert H^{-1}AH-(H^{-1}AH)^\ast\Vert\le c_3\eps^\alpha,
\end{equation}
for all $t\in[0,T]$ and $\eps\in(0,1]$. Hence, combining \eqref{case_2_1}, \eqref{case_2_2} and \eqref{case_2_3} with the energy \eqref{energy_case2} we obtain
\[
\partial_t |W(t,\xi)|^2\le (2\rho'(t)|\xi|_\nu ^{\frac{1}{s}}+c_1\eps^{\alpha-1}+c_2\eps^{\alpha-1}+c_3\eps^{\alpha}|\xi|_\nu )
|W(t,\xi)|^2.
\]
Since $|\xi|_\nu=0$ gives an analytic contribution in view of \eqref{EQ:nu-ests}, it is not restrictive to assume $|\xi|_\nu >0$. 
Hence, by setting $\eps:=|\xi|_\nu^{-1}$ we get
\[
\partial_t |W(t,\xi)|^2\le (2\rho'(t)|\xi|_\nu^{\frac{1}{s}}+c'_1 |\xi|_\nu ^{1-\alpha}+c'_3|\xi|_\nu^{1-\alpha})|W(t,\xi)|^2,
\]
Thus, it follows that for $|\xi|_\nu>0$ we can write, for some constant $C>0$, 
\[
\partial_t |W(t,\xi)|^2\le (2\rho'(t)|\xi|_\nu^{\frac{1}{s}}+C|\xi|_\nu^{1-\alpha})|W(t,\xi)|^2.
\]
At this point taking
\[
\frac{1}{s}> 1-\alpha
\]
and $\rho(t)=\rho(0)-\kappa t$ with $\kappa>0$ to be chosen later, for sufficiently large $|\xi|_\nu$ 
we conclude that 
\[
\partial_t|W(t,\xi)|^2\le 0,
\]
for $t\in[0,T]$ and, for example, without loss of generality, for $|\xi|_\nu \ge 1$. Passing now to $V$ we get
\begin{multline}
\label{last_estimate}
|V(t,\xi)|
=\esp^{-\rho(t)|\xi|_\nu^{\frac{1}{s}}}\frac{1}{\det H(t)}\Vert H(t)\Vert|W(t,\xi)| \\ \le 
\esp^{-\rho(t)|\xi|_\nu^{\frac{1}{s}}}\frac{1}{\det H(t)}\Vert H(t)\Vert|W(0,\xi)|=\\
\esp^{(-\rho(t)+\rho(0))|\xi|_\nu^{\frac{1}{s}}}\frac{\det H(0)}{\det H(t)}\Vert H(t)\Vert \Vert H^{-1}(0)\Vert |V(0,\xi)|,
\end{multline}
where
\[
\frac{\det H(0)}{\det H(t)}\Vert H(t)\Vert \Vert H^{-1}(0)\Vert\le c'.
\]
This is due to the fact that $\det H(t)$ is a bounded function with $\det H(t)=\lambda_2(t)-\lambda_1(t)\ge 2\sqrt{a_0}$ for all $t\in[0,T]$ and $\eps\in(0,1]$, $\Vert H(t)\Vert \le c$ and $\Vert H^{-1}(0)\Vert\le c$ for all $t\in[0,T]$ and $\eps\in(0,1]$. Concluding, there exists a constant $c'>0$ such that
\[
|V(t,\xi)|\le c'\esp^{(-\rho(t)+\rho(0))|\xi|_\nu ^{\frac{1}{s}}}|V(0,\xi)|,
\]
for all $|\xi|_\nu \ge 1$ and $t\in[0,T]$. It is now clear that choosing $\kappa>0$ small enough 
we have that if $|V(0,\xi)|\le c\,\esp^{-\delta|\xi|_\nu^{\frac{1}{s}}}$, $c,\delta>0$, 
the same kind of an estimate holds for $V(t,\xi)$. We finally go back to $\xi$ and $\widehat{v}(t,\xi)$. 
The previous arguments lead to
\[
|\xi|_\nu^2 |\widehat{v}(t,\xi)|^2+|\partial_t \widehat{v}(t,\xi)|^2
\le c'\esp^{(-\rho(t)+\rho(0))|\xi|_\nu^{\frac{1}{s}}}|\xi|_\nu ^2|\widehat{v_0}(\xi)|^2+ 
c'\esp^{(-\rho(t)+\rho(0))|\xi|_\nu^{\frac{1}{s}}}|\widehat{v_1}(\xi)|^2.
\]
Since 
the initial data are both in $\gamma^s_\Lcal(G)$ we obtain that
\begin{equation}
\label{fin_est_case_2}
|\xi|_\nu^2 |\widehat{v}(t,\xi)|^2+|\partial_t \widehat{v}(t,\xi)|^2
\le c'\esp^{
\kappa T|\xi|_\nu^{\frac{1}{s}}}(C_0\esp^{-A_0|\xi|_\nu^{\frac{1}{s}}}+C_1\esp^{-A_1|\xi|_\nu^{\frac{1}{s}}}),
\end{equation}
for suitable constants $C_0,C_1, A_0,A_1>0$ and $\kappa$ small enough, for $t\in[0,T]$ and all $|\xi|_\nu\ge 1$. 
The estimate \eqref{fin_est_case_2} implies that under the hypothesis of Case 2  and for 
\[
1\le s<1+\frac{\alpha}{1-\alpha},
\]
the solution $u$ belongs to $\gamma^s_\Lcal(G)$ in $x$ if the initial data are elements of $\gamma^s_\Lcal(G)$.

\subsection{Case 3: Proof of Theorem \ref{theo_case_3}} 
We now assume that $a(t)\ge 0$ is of class $C^\ell$ on $[0,T]$ with $\ell\ge 2$. 
Adopting the notations of the previous cases we want to study the well-posedness of the system
\eqref{EQ:system}:
it follows that the equation \eqref{eq_xi} can be written as the first order system
\[
\partial_t V(t,\xi)=i|\xi|_\nu A(t)V(t,\xi),
\]
where $V$ is the column vector with entries $V_1$ and $V_2$ and
\[
A(t)=\left(
    \begin{array}{cc}
      0 & 1\\
      a(t) & 0 \\
           \end{array}
  \right).
\]
The initial conditions are  
\[
V(0,\xi)=\left(
    \begin{array}{c}
      i|\xi|_\nu \widehat{v_0}(\xi)\\
      \widehat{v}_{1}(\xi)
     \end{array}
  \right).
\]
This kind of system and the corresponding second order equation have been studied on $\Rn$
in \cite{KS}  and \cite{GR:12} obtaining Gevrey well-posedness for $1\le s<1+\ell/2$ 
and well-posedness in every Gevrey class in case of smooth coefficients. 
The energy is given by a perturbation of the symmetriser, called \emph{quasi-symmetriser}. 
The quasi-symmetriser $Q^{(2)}_\eps$ of $A$ (see \cite{KS}) is defined as
\[
Q^{(2)}_{\eps}(t):=\left(
    \begin{array}{cc}
      2a(t) & 0\\
    0 & 2 \\
           \end{array}
  \right) + 2\eps^2 \left(
    \begin{array}{cc}
      1 & 0\\
      0 & 0 \\
           \end{array}
  \right).
\]
In the sequel we collect a few results which are proven in \cite{GR:12} and are essential for the energy 
estimates below. We refer to Proposition 1, Lemma 1, Lemma 2, and the proof in Section 4.1 in \cite{GR:12}.
\begin{prop}
\label{prop_quasi_sym}
There exists a constant $C_2>0$ such that
\[
C_2^{-1}\eps^{2}|V|^2\le (Q^{(2)}_\eps(t)V,V)\le C_2|V|^2,
\]
and
\[
|((Q_\eps^{(2)}A-A^\ast Q_\eps^{(2)})(t)V,V)|\le C_2\eps (Q_\eps^{2)}(t)V,V),
\]
for all $\eps\in(0,1)$, $t\in[0,T]$ and $V\in\C^2$. In addition the family of matrices $Q^{(2)}_\eps$ is nearly diagonal 
and there exists a constant $C>0$ such that
\[
\int_{0}^T\frac{|(\partial_tQ^{(2)}_\eps(t) V(t,\xi),V(t,\xi))|}{(Q^{(2)}_\eps(t) V(t,\xi), V(t,\xi))}\, dt\\
 \le C\eps^{-2/\ell},
\]
for all $\eps\in(0,1]$, $t\in[0,T]$ and all non-zero continuous functions $V:[0,T]\times\mathbb{R}^n\to \C$.
\end{prop}
Let us introduce the energy $E_\eps(t,\xi)=(Q^{(2)}_\eps(t) V(t,\xi),V(t,\xi))$. By direct computations as in \cite{GR:12} we get
\[
\partial_t E_\eps(t,\xi)=(\partial_t Q_\eps^{(2)}(t)V(t,\xi),V(t,\xi))+i|\xi|_\nu((Q_\eps^{(2)}A-A^\ast Q_\eps^{(2)})(t)V,V)
\]
and therefore by Gronwall lemma and Proposition \ref{prop_quasi_sym}, we get
\begin{equation}
\label{EE_case_3}
E_\eps(t,\xi)\le E_\eps(0,\xi){\rm e}^{c(\eps^{-2/\ell}+\eps|\xi|_\nu)},
\end{equation}
for some constant $c>0$, uniformly in $t$, $\xi$ and $\eps$. By setting $\eps^{-2/\ell}=\eps|\xi|_\nu$ we arrive at
\[
E_\eps(t,\xi)\le E_\eps(0,\xi)C_T{\rm e}^{C_T |\xi|_\nu^{\frac{1}{\sigma}}},
\]
with $\sigma=1+\frac{\ell}{2}$. An application of Proposition \ref{prop_quasi_sym} yields the estimate 
\[
C_2^{-1}\eps^2|V(t,\xi)|^2\le E_\eps(t,\xi)\le E_\eps(0,\xi)C_T{\rm e}^{C_T|\xi|_\nu^{\frac{1}{\sigma}}}
\le C_2|V(0,\xi)|^2C_T{\rm e}^{C_T|\xi|_\nu^{\frac{1}{\sigma}}}
\]
which implies
\[
|V(t,\xi)|\le C|\xi|_\nu^{\frac{k}{2\sigma}}{\rm e}^{C|\xi|_\nu^{\frac{1}{\sigma}}}|V(0,\xi)|,
\]
for some $C>0$, for all $t\in[0,T]$ and for all $\xi$.
We now go back to $\widehat{v}(t,\xi)=\widehat{u}(t,\xi)_{mk}$ to obtain
\[
|\widehat{u}(t,\xi)_{mk}|^2\le C^2 |\xi|_\nu^{\frac{k}{\sigma}}{\rm e}^{2C|\xi|_\nu^{\frac{1}{\sigma}}}
(|\xi|_\nu^2|\widehat{u}_0(\xi)_{mk}|^2+|\widehat{u}_1(\xi)_{mk}|^2)
\]
and recalling that $|\xi|_\nu=\nu_m(\xi)$ and summing over $1\le m,k\le\dxi$, we get
\begin{multline}\label{est_1}
\Vert \widehat{u}(t,\xi)\Vert_{\HS}^2
\le C^2\sum_{1\le m,k\le\dxi} \nu_m(\xi)^{\frac{k}{\sigma}+2}
{\rm e}^{2C \nu_m(\xi)^{\frac{1}{\sigma}}}|\widehat{u}_0(\xi)_{mk}|^2\\
+C^2\sum_{1\le m,k\le\dxi} \nu_m(\xi)^{\frac{k}{\sigma}}{\rm e}^{2C \nu_m(\xi)^{\frac{1}{\sigma}}}|\widehat{u}_1(\xi)_{mk}|^2.
\end{multline}
Recall that the initial data $u_0$ and $u_1$ are elements of $\gamma^s_\Lcal(G)$ and, therefore, 
there exist constants $A',C'>0$ such that
\begin{equation}
\label{est_2}
\Vert \esp^{A'\sigma_{-\Lcal}(\xi)^{\frac{1}{2s}}} \widehat{u}_0(\xi)\Vert_{\HS}\le C',\qquad 
\Vert \esp^{A'\sigma_{-\Lcal}(\xi)^{\frac{1}{2s}}}\widehat{u}_1(\xi)\Vert_{\HS}\le C'.
\end{equation}
Inserting \eqref{est_2} in \eqref{est_1}, taking $s<\sigma$  and $\jp{\xi}$ 
large enough we conclude that there exist constants $C^{''}>0$ such that
\[
\Vert \esp^{A'\sigma_{-\Lcal}(\xi)^{\frac{1}{2s}}} \widehat{u}(t,\xi)\Vert_{\HS}^2\le C^{''},
\]
for all $t\in[0,T]$. 
By Proposition \ref{PROP:subL} it follows that
$$
\sumxi \dxi \Vert \esp^{\frac{A'}{2}\sigma_{-\Lcal}(\xi)^{\frac{1}{2s}}} \widehat{u}(t,\xi)\Vert_{\HS}^2<\infty,
$$
i.e. $u$ belongs to $\gamma^s_\Lcal(G)$ in $x$ provided that
\[
1\le s<\sigma=1+\frac{\ell}{2}.
\]
 
\subsection{Case 4: Proof of Theorem \ref{theo_case_4}}

We finally assume $a(t)\ge 0$ and $a\in C^\alpha([0,T])$ with $0<\alpha<2$. 
The main difference with respect to Case 2 is that now the roots 
$\pm\sqrt{a(t)}$ can coincide and are not H\"older of order $\alpha$ but of order $\alpha/2$. 
For an easy adaptation of the proof of Case 2 in Theorem \ref{theo_case_2} 
we will equivalently assume that
$a\in C^{2\alpha}([0,T])$, $0<\alpha<1$ and that the roots are of class 
$C^\alpha$. 
We now indicate differences with the proof of Theorem \ref{theo_case_2}:
again we look for a solution of the system \eqref{system_A}
of the form
\[
V(t,\xi)={\rm e}^{-\rho(t)|\xi|_\nu^{\frac{1}{s}}}(\det H)^{-1}HW,
\]
where 
$\rho\in C^1([0,T])$ is a real valued function which will be suitably chosen in the sequel, 
\[
H(t)=\left(
    \begin{array}{cc}
      1 & 1 \\
      \lambda_1(t) & \lambda_2(t)  
         \end{array}
  \right)
\]
and, for $\varphi\in C^\infty_c(\mathbb{R})$, $\varphi\ge 0$ with integral $1$,  we set
\begin{equation}
\label{def_lambdaj_4}
\begin{split}
\lambda_1(t)&=(-\sqrt{a}\ast\varphi_\epsilon)(t)+\eps^\alpha,\\
\lambda_2(t)&=(+\sqrt{a}\ast\varphi_\epsilon)(t)+ 2\eps^\alpha.
\end{split}
\end{equation}
Note that $\lambda_1$ and $\lambda_2$ (where the dependence on $\epsilon$ is omitted for the sake of simplicity) are smooth in $t\in[0,T]$ and in addition the following properties hold:
\[
\lambda_2(t)-\lambda_1(t)\ge 2\eps^\alpha,
\]
for all $t\in[0,T]$ and $\eps\in(0,1]$,
\[
|\lambda_1(t)+\sqrt{a(t)}|\le c_1\eps^\alpha
\]
and
\[
|\lambda_2(t)-\sqrt{a(t)}|\le c_2\eps^\alpha,
\]
uniformly in $t$ and $\eps$. Arguing as in Case 2 we arrive at the energy estimate
\begin{multline}
\partial_t |W(t,\xi)|^2\le 2{\rm Re} (\partial_t W(t,\xi),W(t,\xi)) \\
\le 2\rho'(t)|\xi|_\nu^{\frac{1}{s}}
|W(t,\xi)|^2+2\biggl |\frac{\partial_t\det H}{\det H}\biggr||W(t,\xi)|^2\\
+2\Vert H^{-1}\partial_t H\Vert|W(t,\xi)^2|+|\xi|_\nu\Vert H^{-1}AH-(H^{-1}AH)^\ast\Vert |W(t,\xi)|^2.
\end{multline}

We proceed by estimating
\begin{enumerate}
\item $\frac{\partial_t\det H}{\det H}$,
\item $\Vert H^{-1}\partial_t H\Vert$,
\item $\Vert H^{-1}AH-(H^{-1}AH)^\ast\Vert$.
\end{enumerate} 

In analogy with \cite{GR:11} (Subsections 4.1, 4.2, 4.3) 
we get that for all $T>0$ there exist constants $c_1,c_2,c_3>0$ such that
\begin{equation}
\label{case_4_1}
\biggl |\frac{\partial_t\det H}{\det H}\biggr |\le c_1\eps^{1},
\end{equation}
\begin{equation}
\label{case_4_2}
\Vert H^{-1}\partial_t H\Vert\le c_2\eps^{-1},
\end{equation}
\begin{equation}
\label{case_4_3}
\Vert H^{-1}AH-(H^{-1}AH)^\ast\Vert\le c_3\eps^\alpha,
\end{equation}
for all $t\in[0,T]$ and $\eps\in(0,1]$. Hence, combining \eqref{case_4_1}, \eqref{case_4_2} and \eqref{case_4_3} with the previous energy we obtain
\[
\partial_t |W(t,\xi)|^2\le (2\rho'(t)|\xi|_\nu^{\frac{1}{s}}+c_1\eps^{-1}+c_2\eps^{-1}+c_3\eps^{\alpha}|\xi|_\nu)|W(t,\xi)|^2.
\]
Again, it is not restrictive to assume that $|\xi|_\nu>0$. By setting $\eps:=|\xi|_\nu^{-\gamma}$ with
\[
\gamma=\frac{1}{1+\alpha}
\]
we get
\begin{multline*}
\partial_t |W(t,\xi)|^2\le (2\rho'(t)|\xi|_\nu^{\frac{1}{s}}+c'_1|\xi|_\nu^{\gamma}+c'_3|\xi|_\nu^{1-\gamma\alpha})|W(t,\xi)|^2
\\
\le (2\rho'(t)|\xi|_\nu^{\frac{1}{s}}+C|\xi|_\nu^{1/(1+\alpha)})|W(t,\xi)|^2.
\end{multline*}
At this point taking
\[
\frac{1}{s}> \frac{1}{1+\alpha}
\]
and $\rho(t)=\rho(0)-\kappa t$ with $\kappa>0$ to be chosen later, we conclude that 
\[
\partial_t|W(t,\xi)|^2\le 0,
\]
for $t\in[0,T]$ and $|\xi|_\nu\ge 1$. Passing now to $V$ and by the same arguments of Case 2 with
\[
\frac{\det H(0)}{\det H(t)}\Vert H(t)\Vert \Vert H^{-1}(0)\Vert\le c\,\eps^{-\alpha}=c\,|\xi|_\nu^{\gamma\alpha}
=c\,|\xi|_\nu^{\frac{\alpha}{1+\alpha}}
\]
we conclude that there exists a constant $c'>0$ such that
\[
|V(t,\xi)|\le c'|\xi|_\nu^{\frac{\alpha}{1+\alpha}}\esp^{(-\rho(t)+\rho(0))|\xi|_\nu^{\frac{1}{s}}}|V(0,\xi)|,
\]
for all $|\xi|_\nu\ge 1$ and $t\in[0,T]$. 
We finally go back to $\widehat{v}(t,\xi)$ and $\widehat{u}(t,\xi)_{mk}$. We have
\[
\nu_m^2(\xi)|\widehat{u}(t,\xi)_{mk}|^2\le c'\esp^{(-\rho(t)+\rho(0))\nu_m(\xi)^{\frac{1}{s}}}
(\nu_m(\xi)^2
|\widehat{u}_0(\xi)_{mk}|^2+ |\widehat{u}_1(\xi)_{mk}|^2),
\]
with the constant $c'$ independent of $\xi$, $m$ and $k$.
Multiplying by $\esp^{\delta\nu_m(\xi)^{\frac{1}{s}}}$ and 
summing over  $1\leq m,k\leq \dxi$, we get
\begin{multline}\label{EQ:est-HS}
\|\esp^{\delta(\sigma_{-\Lcal}(\xi))^{\frac{1}{2s}}}\sigma_{-\Lcal}(\xi) \widehat{u}(t,\xi)\|_{\HS}^2 \\
\leq 
c' (\|\esp^{(-\rho(t)+\rho(0)+\delta)(\sigma_{-\Lcal}(\xi))^{\frac{1}{2s}}}
\sigma_{-\Lcal}(\xi) \widehat{u}_0(\xi)\|_{\HS}^2
+\|\esp^{(-\rho(t)+\rho(0)+\delta)(\sigma_{-\Lcal}(\xi))^{\frac{1}{2s}}} \widehat{u}_1(\xi)\|_{\HS}^2),
\end{multline}
for any $\delta>0$.
Since the initial data are both in $\gamma^s_\Lcal(G)$, we get that 
$$
\sumxi\dxi (
\|\esp^{(\kappa T+\delta)(\sigma_{-\Lcal}(\xi))^{\frac{1}{2s}}}\sigma_{-\Lcal}(\xi) \widehat{u}_0(\xi)\|_{\HS}^2
+\|\esp^{(\kappa T+\delta)(\sigma_{-\Lcal}(\xi))^{\frac{1}{2s}}} \widehat{u}_1(\xi)\|_{\HS}^2)<\infty
$$
for some $\delta>0$ if $\kappa$ is small enough.
Taking the same sum $\sumxi\dxi$ of the expressions in \eqref{EQ:est-HS}, 
and using Plancherel's formula,  we obtain that
\begin{equation}
\label{fin_est_case_4}
\| \esp^{\delta(-\Lcal)^{\frac{1}{2s}}}\Lcal u(t,\cdot)\|_{L^2}^2=
\sumxi\dxi \|\esp^{\delta(\sigma_{-\Lcal}(\xi))^{\frac{1}{2s}}}\sigma_{-\Lcal}(\xi) \widehat{u}(t,\xi)\|_{\HS}^2
<\infty,
\end{equation}
for $\kappa$ small enough, for $t\in[0,T]$. This completes
the proof of Theorem \ref{theo_case_4}.


\end{document}